\documentclass[a4paper, 11pt]{amsproc}

\usepackage{amsmath, amsthm, amssymb, mathtools, mathrsfs, stmaryrd}
\usepackage{ascmac}
\usepackage{comment}
\usepackage{bm}

\usepackage{hyperref}

\usepackage{graphicx}
\usepackage{here}
\usepackage{time}
\usepackage[abbrev]{amsrefs}

\usepackage{xcolor}
\usepackage[capitalize,nameinlink,noabbrev,nosort]{cleveref}
\hypersetup{
	colorlinks=true,       
	linkcolor=blue,          
	citecolor=blue,        
	filecolor=blue,      
	urlcolor=blue,           
}

\usepackage{fullpage}

\makeatletter
\@namedef{subjclassname@2020}{%
  \textup{2020} Mathematics Subject Classification}
\makeatother

\newtheorem{theoremcounter}{Theorem Counter}[section]

\theoremstyle{definition}
\newtheorem{definition}[theoremcounter]{Definition}

\newtheorem{example}[theoremcounter]{Example}

\theoremstyle{plain}
\newtheorem{lemma}[theoremcounter]{Lemma}
\newtheorem{proposition}[theoremcounter]{Proposition}
\newtheorem{corollary}[theoremcounter]{Corollary}

\newtheorem{theorem}[theoremcounter]{Theorem}

\numberwithin{equation}{section}

\newcommand{\Z}{\mathbb{Z}}
\newcommand{\Q}{\mathbb{Q}}

\DeclareMathOperator{\ReNew}{Re}
\renewcommand{\Re}{\ReNew}

\def\st#1#2{\genfrac{[}{]}{0pt}{}{#1}{#2}}
\def\sts#1#2{\genfrac{\{}{\}}{0pt}{}{#1}{#2}}

\begin{document}


\title{Asymptotic coefficients of multiple zeta functions at the origin and generalized Gregory coefficients}

\author{Toshiki Matsusaka}
\address[Toshiki Matsusaka]{Faculty of Mathematics, Kyushu University,
Motooka 744, Nishi-ku, Fukuoka 819-0395, Japan}
\email{matsusaka@math.kyushu-u.ac.jp} 

\author{Hideki Murahara}
\address[Hideki Murahara]{The University of Kitakyushu,  4-2-1 Kitagata, Kokuraminami-ku, Kitakyushu, Fukuoka, 802-8577, Japan}
\email{hmurahara@mathformula.page}

\author{Tomokazu Onozuka}
\address[Tomokazu Onozuka]{Institute of Mathematics for Industry, Kyushu University 744, Motooka, Nishi-ku, Fukuoka, 819-0395, Japan} \email{t-onozuka@imi.kyushu-u.ac.jp}




\maketitle

\begin{abstract}
    Due to their singularities, multiple zeta functions behave sensitively at non-positive integer points. 
    In this article, we focus on the asymptotic behavior at the origin $(0,\dots, 0)$ and unveil the generating series of the asymptotic coefficients as a generalization of the classical Gregory coefficients. 
    This enables us to reveal the underlying symmetry of the asymptotic coefficients. Additionally, we extend the relationship between the asymptotic coefficients and the Gregory coefficients to include Hurwitz multiple zeta functions.
\end{abstract}


\section{Introduction}
The Euler--Zagier multiple zeta function is defined by
\begin{align*}
 \zeta(s_{1},\dots,s_{r})
 =\sum_{1\le n_{1}<\cdots<n_{r}}
  \frac{1}{n_{1}^{s_{1}}\cdots n_{r}^{s_{r}}},
\end{align*}
where $s_{i}\in\mathbb{C}$ $(i=1,\ldots,r)$ are complex variables.
Matsumoto~\cite{Matsumoto2002} proved that the series is absolutely convergent in the domain
\[
 \{(s_{1},\ldots,s_{r})\in\mathbb{C}^{r}: \Re(s_{l}+\cdots+s_{r})>r-l+1\;\; (l=1,\dots,r)\}.
\]
Akiyama, Egami, and Tanigawa~\cite{AET2001} and Zhao~\cite{Zhao2000} independently proved that $\zeta(s_{1},\dots,s_{r})$ is meromorphically continued to the whole space $\mathbb{C}^{r}$.
Furthermore, all possible singularities of $\zeta(s_{1},\dots,s_{r})$ are located on $s_{l}+\cdots+s_{r}\in\mathbb{Z}_{\le r-l+1}$ for $l=1,\dots,r$.

The asymptotic behavior of multiple zeta functions at non-positive integers was investigated by the third author \cite[Theorem 2]{Onozuka2013}. 
More recently, the second and third authors generalized the results for the Huwritz--Lerch multiple zeta functions \cite[Theorem 1.1]{MuraharaOnozuka2022}. 
First, we recall the result in the special case of multiple zeta functions at the origin $(0, \dots, 0)$, which is the focus of this article.

\begin{theorem}[{\cite{Onozuka2013, MuraharaOnozuka2022}}] \label{theorem:a}
	Let $r \geq 2$ and $\epsilon_1, \dots, \epsilon_r$ be complex numbers. Suppose that $|\epsilon_1|, \dots, |\epsilon_r|$ are sufficiently small with $\epsilon_j \neq 0$, $\epsilon_j + \cdots + \epsilon_r \neq 0$ $(j=1,\dots, r)$, and $|\epsilon_k/(\epsilon_j + \cdots + \epsilon_r)| \ll 1$ as $(\epsilon_1, \dots, \epsilon_r) \to (0, \dots, 0)$ for $1 \leq j \leq k \leq r$. 
    Then we have
	\begin{align}\label{MurOno-asymptotic}
	\begin{split}
		\zeta(\epsilon_1, \dots, \epsilon_r) &= (-1)^r \sum_{d_1, \dots, d_{r-1} \in \{0,1\}} \sum_{(n_1, \dots, n_r) \in S^{(d_1, \dots, d_{r-1})}} \frac{B_{n_1}}{n_1!} \cdots \frac{B_{n_r}}{n_r!} \\
			&\qquad \cdot \prod_{\substack{1 \leq j \leq r-1 \\ d_j=1}} \frac{\epsilon_{j+1} + \cdots + \epsilon_r}{\epsilon_j + \cdots + \epsilon_r} + \sum_{j=1}^r O(|\epsilon_j|),
	\end{split}
	\end{align}
	where $B_n$ is the $n$-th Bernoulli number defined by
	\[
		\sum_{n=0}^\infty B_n \frac{x^n}{n!} = \frac{x e^x}{e^x- 1}
	\]
	and the set $S^{(d_1, \dots, d_{r-1})}$ by
	\[
		S^{(d_1, \dots, d_{r-1})} 
            = \bigcap_{j=1}^{r-1} S_{j,r}^{(d_j)}
	\]
        with
	\begin{align*}
		S_{j,r}^{(d_j)} = \begin{cases}
			\{(n_1, \dots, n_r) \in \Z_{\geq 0}^r : n_1 + \cdots + n_r = r, n_{j+1} + \cdots + n_r \leq r-j\} &\text{if } d_j=0,\\
			\{(n_1, \dots, n_r) \in \Z_{\geq 0}^r : n_1 + \cdots + n_r = r, n_1+\cdots + n_j < j\} &\text{if } d_j=1.
		\end{cases}
	\end{align*}
\end{theorem}

\begin{example}
	For $r=2$ and $r=3$, we have
	\[
		\zeta(\epsilon_1, \epsilon_2) = \frac{1}{3} + \frac{1}{12} \frac{\epsilon_2}{\epsilon_1 + \epsilon_2} + \sum_{j=1}^2 O(|\epsilon_j|)
	\]
	and
	\[
		\zeta(\epsilon_1, \epsilon_2, \epsilon_3) = -\frac{1}{4} - \frac{1}{24} \frac{\epsilon_3}{\epsilon_2+\epsilon_3} -\frac{1}{24} \frac{\epsilon_2 + \epsilon_3}{\epsilon_1+\epsilon_2 + \epsilon_3} -\frac{1}{24} \frac{\epsilon_3}{\epsilon_1+\epsilon_2+\epsilon_3} + \sum_{j=1}^3 O(|\epsilon_j|).
	\]
\end{example}

In this study, we aim to understand better the asymptotic coefficients appearing in this theorem. 
This article is organized as follows. 
In \cref{section2}, we prepare some notations and explain that the coefficients in our asymptotic behavior formula (\cref{theorem:a}) are expressed in terms of primitive elements.
In \cref{section3}, we can observe that the coefficients at the origin contain the Gregory coefficients.
\cref{section4} defines a generalization of the Gregory coefficients and proves that the generalized Gregory coefficients coincide with the primitive coefficients $C^{(1,\ldots,1,0,\ldots,0)}$ in \cref{section2}.
In \cref{section5}, we will show that some of the properties of the classical Gregory coefficients can be generalized to the generalized Gregory coefficients.
%
Finally, in \cref{section6}, we introduce the generalized Gregory polynomials to investigate the asymptotic coefficients of the Hurwitz multiple zeta functions at the origin.

\section{Reduction to the primitive cases}\label{section2}

In advancing the discussion, we will introduce several useful notations. 
\begin{definition}\label{definition-C(S)}
For $r \geq 1$ and a finite set $S \subset \Z_{\geq 0}^r$, we define
\[
	C(S) \coloneqq  (-1)^r \sum_{(n_1, \dots, n_r) \in S} \frac{B_{n_1}}{n_1!} \cdots \frac{B_{n_r}}{n_r!}.
\]
We put $C^{(d_1, \dots, d_{r-1})} \coloneqq  C(S^{(d_1, \dots, d_{r-1})})$ and, specially, $C_{i,r} \coloneqq  C^{(1, \dots, 1, 0, \dots, 0)}$ where the number of $1$ in the superscript is $(i-1)$ for $1 \leq i \leq r$. 
\end{definition}
Under these notations, the asymptotic formula~\eqref{MurOno-asymptotic} can be expressed as
\[
	\zeta(\epsilon_1, \dots, \epsilon_r) = \sum_{d_1, \dots, d_{r-1} \in \{0,1\}} C^{(d_1, \dots, d_{r-1})} \prod_{\substack{1 \leq j \leq r-1 \\ d_j=1}} \frac{\epsilon_{j+1} + \cdots + \epsilon_r}{\epsilon_j + \cdots + \epsilon_r} + \sum_{j=1}^r O(|\epsilon_j|).
\]
In particular, $C_{i,r}$ is the coefficient of $(\epsilon_i + \cdots + \epsilon_r)/(\epsilon_1 + \cdots + \epsilon_r)$. For convenience, we put $S^{()} = \{1\}$, that is,
\[
    C_{1,1} = C^{()} = \zeta(0) = -1/2.
\]
To investigate the coefficients $C^{(d_1, \dots, d_{r-1})}$, we show the following lemmas. For finite sets $S_1, S_2 \subset \Z_{\geq 0}^r$, we say $S_1 \cong S_2$ if there exists a permutation $\sigma \in \mathfrak{S}_r$ such that the map $S_1 \to S_2$ defined by $(n_1, \dots, n_r) \mapsto (n_{\sigma(1)}, \dots, n_{\sigma(r)})$ is bijective.
\begin{lemma}\label{Set-decomp}
	For finite sets $S_1 \subset \Z_{\geq 0}^{r_1}$ and $S_2 \subset \Z_{\geq 0}^{r_2}$, we have the following.
	\begin{itemize}
		\item[(i)] If $r_1 = r_2$ and $S_1 \cong S_2$, then $C(S_1) = C(S_2)$. 
		\item[(ii)] $C(S_1 \times S_2) = C(S_1) C(S_2)$.
	\end{itemize}
\end{lemma}
\begin{proof}
	Both claims immediately follow by definition.
\end{proof}

\begin{theorem}\label{decomposition-formula}
	For any $0 < l < r-1$, we have
	\[
		C^{(d_1, \dots, d_{l-1}, 0, 1, d_{l+2}, \dots, d_{r-1})} = C^{(d_1, \dots, d_{l-1})} C^{(1, d_{l+2}, \dots, d_{r-1})}.
	\]
\end{theorem}
\begin{proof}
	Since
	\[
		S^{(d_1, \dots, d_{l-1}, 0, 1, d_{l+2}, \dots, d_{r-1})} = \bigcap_{j=1}^{l-1} S_{j, r}^{(d_j)} \cap S_{l, r}^{(0)} \cap S_{l+1, r}^{(1)} \cap \bigcap_{j=l+2}^{r-1} S_{j,r}^{(d_j)},
	\]
	any $(n_1, \dots, n_r) \in S^{(d_1, \dots, d_{l-1}, 0, 1, d_{l+2}, \dots, d_{r-1})}$ satisfies that $n_{l+1} + \cdots + n_r \leq r-l$ and $n_1 + \cdots + n_{l+1} < l+1$, that is, $n_1 + \cdots + n_l = l$, $n_{l+1} = 0$, and $n_{l+2} + \cdots + n_r = r-l$. Thus, we see that
	\[
		S^{(d_1, \dots, d_{l-1}, 0, 1, d_{l+2}, \dots, d_{r-1})} \cong S^{(d_1, \dots, d_{l-1})} \times S^{(1, d_{l+2}, \dots, d_{r-1})},
	\]
	which implies the result.
\end{proof}
The decomposition formula reduces the considered coefficients $C^{(d_1, \dots, d_{r-1})}$ to the special cases $C_{i,r}$. We refer to the index $(d_1, \dots, d_{r-1})$ as \emph{primitive} if it satisfies the condition $(d_1, \dots, d_{r-1}) = (1, \dots, 1, 0, \dots, 0)$ according to Sasaki~\cite[Theorem 10]{Sasaki2023}.

\section{The primitive cases}\label{section3}
In the following sections, we focus 
on the coefficients $C_{i,r} = C^{(1,\dots, 1, 0, \dots, 0)}$ $(1 \leq i \leq r)$ for primitive indices, where the numbers of $1$ and $0$ in the superscript are $(i-1)$ and $(r-i)$, respectively. The first few examples are listed below.
\renewcommand{\arraystretch}{1.5}
\begin{table}[H]
\centering
\begin{tabular}{c||c|c|c|c|c|c|c}
	$r \backslash i$ & 1 & 2 & 3 & 4 & 5 & 6 & 7\\ \hline \hline
	1 & $-\frac{1}{2}$ & & & & & & \\ \hline
	2 & $\frac{1}{3}$ & $\frac{1}{12}$ & & & & \\ \hline	
	3 & $-\frac{1}{4}$ & $-\frac{1}{24}$ & $-\frac{1}{24}$ &  &  &  \\ \hline	
	4 & $\frac{1}{5}$ & $\frac{19}{720}$ & $\frac{7}{360}$ & $\frac{19}{720}$ & & \\ \hline	
	5 & $-\frac{1}{6}$ & $-\frac{3}{160}$ & $-\frac{17}{1440}$ & $-\frac{17}{1440}$ & $-\frac{3}{160}$ &  \\ \hline	
	6 & $\frac{1}{7}$ & $\frac{863}{60480}$ & $\frac{41}{5040}$ & $\frac{211}{30240}$ & $\frac{41}{5040}$ & $\frac{863}{60480}$ \\ \hline	
	7 & $-\frac{1}{8}$ & $-\frac{275}{24192}$ & $-\frac{731}{120960}$ & $-\frac{19}{4032}$ & $-\frac{19}{4032}$ & $-\frac{731}{120960}$ & $-\frac{275}{24192}$ \\ 			
\end{tabular}
\caption{Examples of $C_{i,r}$.}
\label{List-Cir}
\end{table}
From this table, we can observe that the column for $i=1$ is given by $C_{1,r} = (-1)^r/(r+1)$, and the column for $i=2$ coincides with the \emph{Gregory coefficients} $C_{2,r} = -G_r$ defined by
\begin{align}\label{Gregory-def}
	\frac{x}{\log(1+x)} = \sum_{n=0}^\infty G_n x^n,
\end{align}
(see also~\cite[A002206]{OEIS}). 
In addition, we observe the symmetric property 
$C_{i+1,r} = C_{r-i+1,r}$ for $1 \leq i \leq r-1$. 
The rest of this article is dedicated to proving these observations and establishing various properties satisfied by $C_{i,r}$ by introducing a generalization of the Gregory coefficients. 
This section shows the following recurrence formula for $C_{i,r}$.
%
\begin{lemma}\label{initial-condition}
	For $r \geq 1$, we have $C_{1,r} = (-1)^r/(r+1)$.
\end{lemma}
\begin{proof}
	By Akiyama and Tanigawa~\cite[Corollary 1]{AkiyamaTanigawa2001}, we have
    \[
        \lim_{\epsilon_1 \to 0} \lim_{\epsilon_2 \to 0} \cdots \lim_{\epsilon_r \to 0} \zeta(\epsilon_1, \dots, \epsilon_r) = \frac{(-1)^r}{r+1},
    \]
    which coincides with $C_{1,r}$.
\end{proof}

\begin{theorem}\label{theorem:Cir-recurrence-rel}
	For $1 \leq i < r$, we have the recurrence relation
	\[
		C_{i+1,r} = C_{i,r} + \sum_{k=r-i+1}^{r-1} C_{r-i+1,k} \cdot C_{r-k,r-k} - \sum_{k=i}^{r-1} C_{i,k} \cdot C_{r-k, r-k}
	\]
	with initial condition $C_{1,r} = (-1)^r/(r+1)$.
\end{theorem}
To prove this theorem, we introduce a generalization of the set $S^{(1,\dots, 1, 0, \dots, 0)}$ and present some useful lemmas.
\begin{definition}
	For integers $i, k, r$ satisfying $1 \leq i \leq k \leq r$, we define the sets
	\[
		S_{i,k,r} \coloneqq  \left\{ (n_1, \dots, n_r) \in \Z_{\geq 0}^r : 
		\begin{aligned}  
			&n_1 + \cdots + n_r = r,\\
			&n_1 + \cdots + n_j < j\quad (1 \leq j < i),\\
			&n_{j+1} + \cdots + n_r \leq r-j\quad (i \leq j < k),\\
			&n_{k+1} + \cdots + n_r \leq r-k-1 \cdots (*),\\
			&n_{j+1} + \cdots + n_r \leq r-j-1\quad (k < j < r)
		\end{aligned} \right\}
	\]
	and $S_{i,k,r}^0$ by replacing the condition $(*)$ with $n_{k+1} + \cdots + n_r = r-k$. In addition, we define $C_{i,k,r} = C(S_{i,k,r})$ and $C_{i,k,r}^0 = C(S_{i,k,r}^0)$.
\end{definition}

By comparing the definitions, we see that $S_{i,r,r} = S^{(1,\dots,1, 0,\dots,0)}$ with $(i-1)$ occurrences of $1$, which implies $C_{i,r} = C_{i,r,r}$.

\begin{lemma}\label{Reversing-n2-nr}
	For $1 \leq i \leq r-1$, the map $S_{i+1,r,r} \to S_{r-i+1, r,r}$ defined by $(n_1, n_2, n_3, \dots, n_r) \mapsto (n_1, n_r, \dots, n_3, n_2)$ is bijective, which implies the symmetric property
	\[
		C_{i+1,r} = C_{r-i+1,r}.
	\]
\end{lemma}
This lemma was shown by Sasaki~\cite[Theorem 8]{Sasaki2023} in more general cases. Here, we provide a direct proof.
\begin{proof}
	For $i+1 \geq 2$, we have
	\[
		S_{i+1,r,r} = \left\{ (0, n_2, \dots, n_r) \in \Z_{\geq 0}^r : 
		\begin{aligned}  
			&n_2 + \cdots + n_r = r,\\
			&n_2 + \cdots + n_j \leq j-1\quad (1 < j \leq i),\\
			&n_{j+1} + \cdots + n_r \leq r-j\quad (i+1 \leq j < r)
		\end{aligned} \right\}.
	\]
	With this expression, the map reversing the order of $n_2, \dots, n_r$ is a bijection from $S_{i+1,r,r}$ to $S_{r-i+1,r,r}$.
\end{proof}

\begin{lemma}\label{Cir-C+C0}
	For $1 \leq i < k \leq r$, we have $S_{i,k,r} = S_{i,k-1,r} \sqcup S_{i,k-1,r}^0$. Thus, for $1 \leq i \leq r$,
	\[
		C_{i,r} = C_{i,i,r} + \sum_{k=i}^{r-1} C_{i,k,r}^0.
	\]
\end{lemma}
\begin{proof}
	By dividing the inequality condition $n_k + \cdots + n_r \leq r-(k-1)$ for $j=k-1$ in the definition of $S_{i,k,r}$ into two conditions
	\begin{itemize}
		\item $n_k + \cdots + n_r = r-(k-1)$,
		\item $n_k + \cdots + n_r \leq r-k$,
	\end{itemize}
	the first claim immediately follows. By applying it repeatedly, we obtain that
	\[
		C_{i,r,r} = C_{i,r-1,r} + C_{i,r-1,r}^0 = \cdots = C_{i,i,r} + \sum_{k=i}^{r-1} C_{i,k,r}^0,
	\]
	which concludes the proof.
\end{proof}

\begin{lemma}\label{C0-product}
	For $1 \leq i \leq k \leq r$, we have
	\[
		C_{i,k,r}^0 = C_{i,k} \cdot C_{r-k,r-k}.
	\]
\end{lemma}
\begin{proof}
	By definition, we have
	\begin{align*}
		S_{i,k,r}^0 &= \left\{ (n_1, \dots, n_k) \in \Z_{\geq 0}^k : 
		\begin{aligned}  
			&n_1 + \cdots + n_k = k,\\
			&n_1 + \cdots + n_j < j \quad (1 \leq j < i),\\
			&n_{j+1} + \cdots + n_k \leq k-j \quad (i \leq j < k),
		\end{aligned} \right\} \\
		&\qquad \times 		
		\left\{ (n_{k+1}, \dots, n_r) \in \Z_{\geq 0}^{r-k} : 
		\begin{aligned} 	
			&n_{k+1} + \cdots + n_r = r-k,\\
			&n_{j+1} + \cdots + n_r < r-j \quad (k < j \leq r-1)
		\end{aligned} \right\}\\
		&\cong S_{i,k,k} \times S_{r-k, r-k, r-k},
	\end{align*}
	where the reverse operation gives the last bijection. Therefore, $C_{i,k,r}^0 = C_{i,k,k} \cdot C_{r-k,r-k,r-k} = C_{i,k} \cdot C_{r-k,r-k}$.
\end{proof}

\begin{proof}[Proof of \cref{theorem:Cir-recurrence-rel}]
	First, we note that
	\[
		S_{i,i,r} \coloneqq  \left\{ (n_1, \dots, n_r) \in \Z_{\geq 0}^r : 
		\begin{aligned}  
			&n_1 + \cdots + n_r = r,\\
			&n_1 + \cdots + n_j \leq j-1 \quad (1 \leq j < i),\\
			&n_{j+1} + \cdots + n_r \leq r-j-1 \quad (i \leq j < r)
		\end{aligned} \right\}
	\]
	satisfies the symmetric property $S_{i,i,r} \cong S_{r-i+1, r-i+1,r}$, that is, $C_{i,i,r} = C_{r-i+1, r-i+1, r}$. By it with \cref{Reversing-n2-nr} and \cref{Cir-C+C0}, we have
	\begin{align*}
		C_{i+1,r} &= C_{r-i+1,r} = C_{r-i+1, r-i+1, r} + \sum_{k=r-i+1}^{r-1} C_{r-i+1,k,r}^0\\
			&= C_{i,i,r}+ \sum_{k=r-i+1}^{r-1} C_{r-i+1,k,r}^0.
	\end{align*}
	By comparing it with \cref{Cir-C+C0} and eliminating $C_{i,i,r}$, we obtain 
	\[
		C_{i,r} - \sum_{k=i}^{r-1} C_{i,k,r}^0 = C_{i+1,r} - \sum_{k=r-i+1}^{r-1} C_{r-i+1,k,r}^0.
	\]
	Finally, \cref{C0-product} implies the desired result.
\end{proof}

\section{Generalized Gregory coefficients}\label{section4}
This section aims to express the recurrence relation provided in \cref{theorem:Cir-recurrence-rel} in generating series.
We use $\log^m x = (\log x)^m$ to simplify the notation.
\begin{definition}\label{definition:Generalized-Gregory}
    We define the \emph{Generalized Gregory coefficients} $G_{m,n}$ by the generating series	
    \[
		\mathcal{G}(x,y) \coloneqq  \sum_{m, n \geq 0} G_{m,n} x^m y^n \coloneqq  \frac{y \log^2(1+x) - x \log^2(1+y)}{\log(1+x) - \log(1+y)}.
	\]
\end{definition}
The first few examples are listed below.
\renewcommand{\arraystretch}{1.5}
\begin{table}[H]
\centering
\begin{tabular}{c||c|c|c|c|c|c|c}
	$m \backslash n$ & 0 & 1 & 2 & 3 & 4 & 5 & 6\\ \hline \hline
	0 & 0 & 0 & 0 & $0$ & 0 & 0 & 0\\ \hline	
	1 & 0 & 1 & $-\frac{1}{2}$ & $\frac{1}{3}$ & $-\frac{1}{4}$ & $\frac{1}{5}$ & $-\frac{1}{6}$ \\ \hline	
	2 & 0 & $-\frac{1}{2}$ & $\frac{1}{12}$ & $-\frac{1}{24}$ & $\frac{19}{720}$ & $-\frac{3}{160}$ & $\frac{863}{60480}$ \\ \hline	
	3 & 0 & $\frac{1}{3}$ & $-\frac{1}{24}$ & $\frac{7}{360}$ & $-\frac{17}{1440}$ & $\frac{41}{5040}$ & $-\frac{731}{120960}$ \\ \hline	
	4 & 0 & $-\frac{1}{4}$ & $\frac{19}{720}$ & $-\frac{17}{1440}$ & $\frac{221}{30240}$ & $-\frac{19}{4032}$ & $\frac{12499}{3628800}$ \\ \hline	
	5 & 0 & $\frac{1}{5}$ & $-\frac{3}{160}$ & $\frac{41}{5040}$ & $-\frac{19}{4032}$ & $\frac{2843}{907200}$ & $-\frac{469}{207360}$ \\ \hline
	6 & 0 & $-\frac{1}{6}$ & $\frac{863}{60480}$ & $-\frac{731}{120960}$ & $\frac{12499}{3628800}$ & $-\frac{469}{207360}$ & $\frac{386831}{239500800}$
\end{tabular}
\caption{Examples of $G_{m,n}$.}
\end{table}
We immediately see the symmetric property $G_{m,n} = G_{n,m}$ by definition. Furthermore, upon comparing the list with \cref{List-Cir}, we observe that $C_{i,r} = G_{i, r-i+2}$. First, we show the coincidence.
\begin{theorem}\label{Coefficient=Generalized-Gregory}
	For $1 \leq i \leq r$, we have $C_{i,r} = G_{i,r-i+2}$.
\end{theorem}
To prove this theorem, we prepare some lemmas.
\begin{lemma}\label{Generalized-Gregory-Gregory}
	We have
	\[
		\sum_{m=1}^\infty G_{m,1} x^m = \log(1+x) = \sum_{m=1}^\infty \frac{(-1)^{m-1}}{m} x^m
	\]
	and
	\[
		1 - \sum_{m=1}^\infty G_{m,2} x^m = \frac{x}{\log(1+x)},
	\]
	which imply that $G_{m,1} = (-1)^{m-1}/m$ and $G_{m,2} = -G_m$, where $G_m$ is the $m$-th Gregory coefficient defined in \eqref{Gregory-def}.
\end{lemma}
\begin{proof}
	They follow from $\mathcal{G}_y(x,0)$ and $\mathcal{G}_{yy}(x,0)/2$.
\end{proof}

\begin{definition}\label{Gtilde-def}
	For $m, n \geq 1$, we define the sum-product
	\[
		\widetilde{G}_{m,n} \coloneqq  G_{m,n} + \sum_{j=0}^{m-2} G_{n-1, j+2} \cdot G_{m-j-1, 2}.
	\]
	For $m=1$, we put $\widetilde{G}_{1,n} = G_{1,n}$.
\end{definition}
Then we can observe a symmetric property of $\widetilde{G}_{m,n}$, that is, $\widetilde{G}_{m,n} = \widetilde{G}_{n-1,m+1}$ from the following table.
\renewcommand{\arraystretch}{1.5}
\begin{table}[H]
\centering
\begin{tabular}{c||c|c|c|c|c|c}
	$m \backslash n$ & 1 & 2 & 3 & 4 & 5 & 6\\ \hline \hline
	1 & 1 & $-\frac{1}{2}$ & $\frac{1}{3}$ & $-\frac{1}{4}$ & $\frac{1}{5}$ & $-\frac{1}{6}$ \\ \hline	
	2 & $-\frac{1}{2}$ & $\frac{1}{3}$ & $-\frac{1}{12}$ & $\frac{17}{360}$ & $-\frac{23}{720}$ & $\frac{143}{6048}$ \\ \hline	
	3 & $\frac{1}{3}$ & $-\frac{1}{4}$ & $\frac{17}{360}$ & $-\frac{1}{40}$ & $\frac{491}{30240}$ & $-\frac{353}{30240}$ \\ \hline	
	4 & $-\frac{1}{4}$ & $\frac{1}{5}$ & $-\frac{23}{720}$ & $\frac{491}{30240}$ & $-\frac{311}{30240}$ & $\frac{3293}{453600}$ \\ \hline	
	5 & $\frac{1}{5}$ & $-\frac{1}{6}$ & $\frac{143}{6048}$ & $-\frac{353}{30240}$ & $\frac{3293}{453600}$ & $-\frac{2293}{453600}$ \\ \hline
	6 & $-\frac{1}{6}$ & $\frac{1}{7}$ & $-\frac{373}{20160}$ & $\frac{5401}{604800}$ & $-\frac{2207}{403200}$ & $\frac{902863}{239500800}$
\end{tabular}
\caption{Examples of $\widetilde{G}_{m,n}$.}
\end{table}

\begin{lemma}\label{G-til-symmetry}
	We have $\widetilde{G}_{m,1} = G_{m,1} = (-1)^{m-1}/m$ and
	\begin{align*}
		\widetilde{\mathcal{G}}(x,y) &\coloneqq  \sum_{m=1}^\infty \sum_{n=1}^\infty \widetilde{G}_{m,n+1} x^m y^n\\
			&= \frac{y \log(1+x) - x \log(1+y)}{\log(1+x) - \log(1+y)} \left(\frac{y \log(1+x) + x \log(1+y)}{xy} - 1 \right),
	\end{align*}
	which implies the symmetric property $\widetilde{G}_{m,n} = \widetilde{G}_{n-1,m+1}$ for $m \geq 1$ and $n \geq 2$.
\end{lemma}
\begin{proof}
	The first claim immediately follows from the definition. For the second claim, by \cref{Generalized-Gregory-Gregory}, we have
	\begin{align*}
		\sum_{m=1}^\infty \sum_{n=1}^\infty &\widetilde{G}_{m,n+1} x^m y^n = \sum_{m=1}^\infty \sum_{n= 1}^\infty G_{m,n+1} x^m y^n + \sum_{m=2}^\infty \sum_{n= 1}^\infty \sum_{j=0}^{m-2} G_{n, j+2} \cdot G_{m-j-1, 2} x^m y^n\\
		&= y^{-1} \bigg(\sum_{m=1}^\infty \sum_{n=1}^\infty G_{m,n} x^m y^n - \sum_{m=1}^\infty G_{m,1} x^m y \bigg)\\
		&\qquad + x^{-1} \bigg(\sum_{j=1}^\infty \sum_{n=1}^\infty G_{j,n} x^j y^n - \sum_{n=1}^\infty G_{1,n} x y^n \bigg) \sum_{m=1}^\infty G_{m, 2} x^m\\
		&= y^{-1} \bigg(\mathcal{G}(x,y) - y \log(1+x) \bigg) + x^{-1} \bigg(\mathcal{G}(x,y) - x \log(1+y) \bigg) \bigg(1 - \frac{x}{\log(1+x)} \bigg).
	\end{align*}
	By applying the definition of $\mathcal{G}(x,y)$, we obtain the desired result.
\end{proof}
The symmetric property of $\widetilde{G}_{m,n}$ captures the recurrence formula for $C_{i,r}$ as proven in~\cref{theorem:Cir-recurrence-rel}, which shows the coincidence $C_{i,r} = G_{i,r-i+2}$.
\begin{proof}[Proof of \cref{Coefficient=Generalized-Gregory}]
	By referring to \cref{theorem:Cir-recurrence-rel}, it suffices to show that $G_{1,r} = (-1)^{r-1}/r$ and
	\[
		G_{i+1,r-i+1} = G_{i,r-i+2} + \sum_{k=r-i+1}^{r-1} G_{r-i+1,k-r+i+1} \cdot G_{r-k,2} - \sum_{k=i}^{r-1} G_{i,k-i+2} \cdot G_{r-k, 2}.
	\]
	The first claim is already proven in \cref{Generalized-Gregory-Gregory}. The second claim can be rewritten as
	\[
		G_{i+1,j-1} = G_{i,j} + \sum_{k=0}^{i-2} G_{j-1,k+2} \cdot G_{i-k-1,2} - \sum_{k=0}^{j-3} G_{i,k+2} \cdot G_{j-k-2, 2}
	\]
	by setting $j = r - i +2 \geq 2$. This is also equivalent to the expression $\widetilde{G}_{i,j} = \widetilde{G}_{j-1, i+1}$ for $i \geq 1$ and $j \geq 2$, which is already shown in \cref{G-til-symmetry}.
\end{proof}

\section{Some properties of Generalized Gregory coefficients $G_{m,n}$}\label{section5}
This section shows various formulas satisfied by $G_{m,n}$, that is $C_{i,r}$, derived from its generating series.
First, we establish a generalization of the integral expression for the Gregory coefficients,
\[
	G_{m,2} = - \int_0^1 {t \choose m} dt.
\]
Here $m \geq 1$, and ${t \choose m}$ represents the binomial coefficient. To derive the expression for $G_{m,n}$, we recall the Stirling polynomials of the first kind, previously investigated by B\'{e}nyi and the first author~\cite{BenyiMatsusaka2022} from a combinatorial perspective.
\begin{definition}[{\cite[Proposition 2.21]{BenyiMatsusaka2022}}]\label{def-Stirling-poly-1st}
	For integers $n \geq m \geq 0$, we define \emph{the Stirling polynomials of the first kind} $\st{n}{m}_x \in \Z[x]$ by
	\[
		\sum_{m=0}^n \st{n}{m}_x y^m \coloneqq  (x+y)(x+y+1) \cdots (x+y+n-1).
	\]
\end{definition}
The first few examples of the Stirling polynomials of the first kind are listed below.
\renewcommand{\arraystretch}{1.5}
\begin{table}[H]
\centering
\begin{tabular}{c||c|c|c|c|c}
	$n \backslash m$ & 0 & 1 & 2 & 3 & 4 \\ \hline \hline
	0 & 1 & 0 & 0 & $0$ & 0  \\ \hline	
	1 & $x$ & 1 & 0 & 0 & 0  \\ \hline	
	2 & $x^2+x$ & $2x+1$ & 1 & 0 & 0 \\ \hline	
	3 & $x^3 + 3x^2 + 2x$ & $3x^2+6x+2$ & $3x+3$ & $1$ & 0 \\ \hline	
	4 & $x^4+6x^3+11x^2+6x$ & $4x^3+18x^2+22x+6$ & $6x^2+18x+11$ & $4x+6$ & $1$
\end{tabular}
\caption{Examples of $\st{n}{m}_x$.}
\end{table}

\begin{lemma}\label{St1-poly-generating}
	The polynomials satisfy the recursion
	\[
		\st{n+1}{m}_x = \st{n}{m-1}_x + (x+n) \st{n}{m}_x
	\]
	with the initial conditions $\st{0}{0}_x = 1$, $\st{n}{0}_x = x(x+1) \cdots (x+n-1)$, and $\st{0}{m}_x = 0$. We also have
	\[
		\mathcal{S}_m(x,t) \coloneqq  (1-t)^{-x} \frac{(-1)^m}{m!} \log^m(1-t) = \sum_{n=m}^\infty \st{n}{m}_x \frac{t^n}{n!}.
	\]
\end{lemma}
\begin{proof}
	The first claim immediately follows from \cref{def-Stirling-poly-1st} and a simple observation $x+y+n = y + (x+n)$. The second claim follows from 
	\[
		\bigg((1-t) \frac{d}{dt} - x\bigg) \mathcal{S}_m(x,t) = \mathcal{S}_{m-1}(x,t),
	\]
	$\mathcal{S}_m(x,0) = 0$ for $m \geq 1$, and $\mathcal{S}_0(x,t) = (1-t)^{-x} = \sum_{n=0}^\infty x(x+1) \cdots (x+n-1) t^n/n!$.
\end{proof}

\begin{theorem}\label{Gmn-integral-exp}
	For $m \geq 1$ and $n \geq 2$, we have
	\[
		G_{m,n} = \frac{2 (-1)^{n-1}}{n!} \int_0^1 {t \choose m} \st{n}{2}_{t-1} dt.
	\]
\end{theorem}
\begin{proof}
	By \cref{Generalized-Gregory-Gregory} and \cref{St1-poly-generating}, we have
	\begin{align*}
		&\sum_{m=1}^\infty G_{m,1} x^m y + \sum_{m=1}^\infty \sum_{n=2}^\infty \left( \frac{2(-1)^{n-1}}{n!} \int_0^1 {t \choose m} \st{n}{2}_{t-1} dt \right) x^m y^n\\
			&\qquad = y \log(1+x) - 2 \int_0^1 \left(\sum_{m=1}^\infty {t \choose m} x^m \right) \left(\sum_{n=2}^\infty \st{n}{2}_{t-1} \frac{(-y)^n}{n!} \right) dt\\
			&\qquad = y \log(1+x) - \log^2(1+y) \int_0^1 ((1+x)^t - 1) (1+y)^{1-t} dt\\
			&\qquad = y \log(1+x) - \log^2(1+y) \left(\frac{x-y}{\log(1+x) - \log(1+y)} - \frac{y}{\log(1+y)} \right),
	\end{align*}
	which coincides with $\mathcal{G}(x,y)$.
\end{proof}
We also note that the Stirling polynomials of the first kind with a lower index of 1 are closely related to the sum-product $\widetilde{G}_{m,n}$ defined in \cref{Gtilde-def} in the following sense.
\begin{proposition}
	We define $G_{m,n}^{(1)}$ by the generating series
	\[
		\sum_{m, n \geq 0} G_{m,n}^{(1)} x^m y^n = \frac{y \log(1+x) - x \log(1+y)}{\log(1+x) - \log(1+y)}
	\]
	appeared as a part of $\widetilde{\mathcal{G}}(x,y)$ in \cref{G-til-symmetry}. For $m, n \geq 1$, we have
	\[
		G_{m,n}^{(1)} = \frac{(-1)^n}{n!} \int_0^1 {t \choose m} \st{n}{1}_{t-1} dt.
	\]
\end{proposition}
\begin{proof}
	The proof is entirely similar to that of \cref{Gmn-integral-exp}.
\end{proof}

\renewcommand{\arraystretch}{1.5}
\begin{table}[H]
\centering
\begin{tabular}{c||c|c|c|c|c|c|c}
	$m \backslash n$ & 0 & 1 & 2 & 3 & 4 & 5 & 6\\ \hline \hline
	0 & 0 & 0 & 0 & $0$ & 0 & 0 & 0\\ \hline	
	1 & 0 & $-\frac{1}{2}$ & $\frac{1}{12}$ & $-\frac{1}{24}$ & $\frac{19}{720}$ & $-\frac{3}{160}$ & $\frac{863}{60480}$ \\ \hline	
	2 & 0 & $\frac{1}{12}$ & $0$ & $-\frac{1}{720}$ & $\frac{1}{720}$ & $-\frac{5}{4032}$ & $\frac{11}{10080}$ \\ \hline	
	3 & 0 & $-\frac{1}{24}$ & $-\frac{1}{720}$ & $\frac{1}{720}$ & $-\frac{17}{15120}$ & $\frac{37}{40320}$ & $-\frac{103}{134400}$ \\ \hline	
	4 & 0 & $\frac{19}{720}$ & $\frac{1}{720}$ & $-\frac{17}{15120}$ & $\frac{13}{15120}$ & $-\frac{2473}{3628800}$ & $\frac{2027}{3628800}$ \\ \hline	
	5 & 0 & $-\frac{3}{160}$ & $-\frac{5}{4032}$ & $\frac{37}{40320}$ & $-\frac{2473}{3628800}$ & $\frac{643}{1209600}$ & $-\frac{25813}{59875200}$ \\ \hline
	6 & 0 & $\frac{863}{60480}$ & $\frac{11}{10080}$ & $-\frac{103}{134400}$ & $\frac{2027}{3628800}$ & $-\frac{25813}{59875200}$ & $\frac{4157}{11975040}$
\end{tabular}
\caption{Examples of $G_{m,n}^{(1)}$.}
\end{table}
Finally, we provide two recurrence formulas for $G_{m,n}$ with symmetry.
\begin{theorem}\label{Gmn-rec-1}
	For $m, n \geq 2$, we have
	\[
		\sum_{k=2}^m \sum_{l=2}^n k! \sts{m}{k} l! \sts{n}{l} G_{k,l} = 1 - \frac{m! n!}{(m+n-1)!},
	\]
	where $\sts{n}{m}$ is the Stirling number of the second kind defined by the recurrence formula 
	\[
		\sts{n+1}{m} = \sts{n}{m-1} + m \sts{n}{m}
	\]
	with the initial conditions $\sts{0}{0} = 1, \sts{n}{0} = \sts{0}{m} = 0$ $(n, m \neq 0)$.
\end{theorem}
\begin{proof}
	By combining \cref{definition:Generalized-Gregory} and \cref{Generalized-Gregory-Gregory}, we have
	\begin{align}\label{Gmn-geq2}
		\sum_{m, n \geq 2} G_{m,n} x^m y^n = xy - \frac{(x-y) \log(1+x) \log(1+y)}{\log(1+x) - \log(1+y)}.
	\end{align}
	We note that the Stirling numbers satisfy
	\[
		\frac{(e^t-1)^m}{m!} = \sum_{n=m}^\infty \sts{n}{m} \frac{t^n}{n!},
	\]
	(see~\cite[Proposition 2.6]{ArakawaIbukiyamaKaneko2014}). Then the generating series
	\[
		\sum_{m,n \geq 2} G_{m,n} (e^x-1)^m (e^y-1)^n = (e^x-1)(e^y-1) - \frac{(e^x-e^y) x y}{x-y}
	\]
	implies
	\begin{align*}
		\sum_{m, n \geq 2} \left(\sum_{k=2}^m \sum_{l=2}^n G_{k,l} k!  \sts{m}{k} l! \sts{n}{l} \right) \frac{x^m}{m!} \frac{y^n}{n!} &= \sum_{m, n \geq 1} \frac{x^m}{m!} \frac{y^n}{n!} - \sum_{m,n \geq 1} \frac{x^m y^n}{(m+n-1)!} \\
			&= \sum_{m,n \geq 2} \left(1 - \frac{m!n!}{(m+n-1)!} \right) \frac{x^m}{m!} \frac{y^n}{n!},
	\end{align*}
	which concludes the proof.
\end{proof}

\begin{theorem}\label{Gmn-rec-2}
	For $m, n \geq 2$, we have
	\[
		\sum_{k=2}^m \sum_{l=2}^n \frac{(-1)^{k+l} G_{k,l}}{m-k + n-l + 1} = \frac{1}{m+n-1} - \frac{1}{mn}.
	\]
\end{theorem}
\begin{proof}
	By combining the fact
	\[
		\frac{\log(1+x) - \log(1+y)}{x-y} = \sum_{m,n \geq 0} \frac{(-1)^{m+n}}{m+n+1} x^m y^n
	\]
	and \eqref{Gmn-geq2}, we have
	\begin{align*}
		&\sum_{m, n \geq 2} (-1)^{m+n} \sum_{k=2}^m \sum_{l=2}^n \frac{(-1)^{k+l} G_{k,l}}{m-k + n-l + 1} x^m y^n\\
			&\qquad = \sum_{k, l \geq 2} G_{k,l} x^k y^l \sum_{m, n \geq 0} \frac{(-1)^{m+n}}{m + n + 1} x^m y^n\\
			&\qquad = \left(xy - \frac{(x-y)\log(1+x)\log(1+y)}{\log(1+x) - \log(1+y)} \right) \frac{\log(1+x) - \log(1+y)}{x-y}\\
			&\qquad = \frac{\log(1+x) - \log(1+y)}{x-y} xy - \log(1+x) \log(1+y)\\
			&\qquad = \sum_{m, n \geq 1} (-1)^{m+n} \left(\frac{1}{m+n-1} - \frac{1}{mn} \right) x^m y^n,
	\end{align*}
	which concludes the proof.
\end{proof}
These theorems provide generalizations of the recurrence formulas for the original Gregory coefficients $G_n$ defined in \eqref{Gregory-def} by
\[
    \frac{x}{\log(1+x)} = \sum_{n=0}^\infty G_n x^n = 1 + \frac{1}{2}x - \frac{1}{12} x^2 + \frac{1}{24} x^3 - \frac{19}{720}x^4 + \frac{3}{160}x^5 + \cdots.
\]
\begin{corollary}
    For $n \geq 0$, we have
    \[
        \sum_{l=0}^n l! \sts{n}{l} G_l = \frac{1}{n+1}
    \]
    and
    \[
        \sum_{l=0}^n \frac{(-1)^l G_l}{n-l+1} = \delta_{n,0}.
    \]
\end{corollary}
\begin{proof}
    The claims immediately follow from applying \cref{Gmn-rec-1} and \cref{Gmn-rec-2} for $m=2$. Here we note that $G_n = -G_{2,n}$ for $n \geq 1$ and $G_0 = 1$.
\end{proof}

\section{Asymptotic coefficients of Hurwitz multiple zeta functions at the origin}\label{section6}
In the previous discussion, we considered the coefficients of the asymptotic behavior of the multiple zeta functions at the origin. 
We showed that the primitive elements are identical to the generalized Gregory coefficients and found several formulas.
In this last section, we generalize them for the Hurwitz multiple zeta functions defined by
\begin{align*}
    \zeta(s_{1},\ldots,s_{r};a_{1},\ldots,a_r) \coloneqq \sum_{0\le m_{1},\ldots,m_{r}}\frac{1}{(m_{1}+a_{1})^{s_{1}}\cdots(m_{1}+\cdots+m_{r}+a_{1}+\cdots+a_{r})^{s_{r}}},
\end{align*}
where $-\pi<\arg(m_1+\cdots+m_j+a_1+\cdots+a_j)\le\pi$ $(1\le j\le r)$. 

According to \cite[Theorem 1.1]{MuraharaOnozuka2022}, \cref{theorem:a} was generalized to
\begin{align*}
	&\zeta(\epsilon_1, \dots, \epsilon_r; a_1, \dots, a_r)\\
	&= \sum_{d_1, \dots, d_{r-1} \in \{0,1\}} C^{(d_1, \dots, d_{r-1})}(a_1, \dots, a_r) \prod_{\substack{1 \leq j \leq r-1 \\ d_j=1}} \frac{\epsilon_{j+1} + \cdots + \epsilon_r}{\epsilon_j + \cdots + \epsilon_r} + \sum_{j=1}^r O(|\epsilon_j|),
\end{align*}
where
\[
	C^{(d_1, \dots, d_{r-1})}(a_1, \dots, a_r) \coloneqq (-1)^r \sum_{(n_1, \dots, n_r) \in S^{(d_1, \dots, d_{r-1})}} \frac{B_{n_1}(a_1)}{n_1!} \cdots \frac{B_{n_r}(a_r)}{n_r!}
\]
and $B_n(a)$ is the $n$-th Bernoulli polynomial defined by
\[
	\sum_{n=0}^\infty B_n(a) \frac{x^n}{n!} = \frac{x e^{ax}}{e^x-1}.
\]
It immediately follows that $C^{(d_1, \dots, d_{r-1})} (1, \dots, 1) = C^{(d_1, \dots, d_{r-1})}$. From the perspective of the symmetric property shown in \cref{Reversing-n2-nr}, in the remainder of this section, we consider only the diagonal case of $a = a_1 = \cdots = a_r$ and parallelly generalize the contents from \cref{section2} to \cref{section4}.

First, natural generalizations of \cref{Set-decomp} and \cref{decomposition-formula} hold with the exact same proof. Thus, in this case as well, our focus is reduced to the primitive case
\[
	C_{i,r}(a) \coloneqq C^{(1,\dots, 1, 0, \dots, 0)}(a, \dots, a),
\]
where the number of $1$ in the superscript is $(i-1)$ for $1 \leq i \leq r$. When distinguishing variables, we express $C_{i,r}(a_1, \dots, a_r) = C^{(1,\dots,1,0,\dots,0)}(a_1, \dots, a_r)$. The first few examples are listed below.

\renewcommand{\arraystretch}{1.5}
\begin{table}[H]
\centering
\begin{tabular}{c||c|c|c}
	$r \backslash i$ & 1 & 2 & 3 \\ \hline \hline
	1 & $-a+\frac{1}{2}$ &  \\ \hline
	2 & $\frac{3}{2}a^2 - \frac{3}{2}a + \frac{1}{3}$ & $\frac{1}{2} a^2 - \frac{1}{2} a + \frac{1}{12}$ \\ \hline	
	3 & $-\frac{8}{3}a^3 + 4a^2 - \frac{11}{6}a + \frac{1}{4}$ & $-\frac{2}{3}a^3 + a^2 - \frac{5}{12}a + \frac{1}{24}$ & $-\frac{2}{3}a^3 + a^2 - \frac{5}{12}a + \frac{1}{24}$ 
\end{tabular}
\caption{Examples of $C_{i,r}(a)$.}
\label{List-Cir(a)}
\end{table}
Then, we have the following recurrence relation.
\begin{theorem}[A generalization of \cref{theorem:Cir-recurrence-rel}]\label{Cir(a)-rec}
	For $1 \leq i < r$, we have the recurrence relation
	\[
		C_{i+1,r}(a) = C_{i,r}(a) + \sum_{k=r-i+1}^{r-1} C_{r-i+1,k}(a) \cdot C_{r-k, r-k}(a) - \sum_{k=i}^{r-1} C_{i,k}(a) \cdot C_{r-k,r-k}(a)
	\]
	with initial values $C_{1,r}(a)$.
\end{theorem}
\begin{proof}
	Since the proof of \cref{theorem:Cir-recurrence-rel} is given in terms of $S_{i,k,r}$, the exact same proof works as well. 
\end{proof}

Next, to generalize \cref{definition:Generalized-Gregory}, we introduce a generalization of the logarithm function.
\begin{definition}
    We define a formal power series
    \[
        L(a,t) = \sum_{n=1}^\infty \lambda_n(a) t^n \in \Q[a]\llbracket t \rrbracket
    \]
    by $e^{L(a,t)}(1 - te^{(a-1)L(a,t)})=1$. In other words, $L(a,t)$ is the inverse function of $e^{-at}(e^t-1)$.
\end{definition}
Note that $L(1,t) = -\log(1-t)$, that is, $\lambda_n(1) = 1/n$, and the first few examples are given by
\begin{align*}
	e^{L(a,t)} (1 - t e^{(a-1)L(a,t)}) = 1 + (\lambda_1(a)-1) t + (2\lambda_2(a) + \lambda_1(a)^2 - 2a \lambda_1(a))\frac{t^2}{2!} + \cdots,
\end{align*}
that is,
\begin{align*}
	\lambda_1(a) &= 1,\\
	\lambda_2(a) &= a- \frac{1}{2} = \frac{(2a-1)}{2!},\\
	\lambda_3(a) &= \frac{3}{2} a^2 - \frac{3}{2} a + \frac{1}{3} = \frac{(3a-1)(3a-2)}{3!},\\
	\lambda_4(a) &= \frac{8}{3}a^3 - 4a^2 + \frac{11}{6}a - \frac{1}{4} = \frac{(4a-1)(4a-2)(4a-3)}{4!}.
\end{align*}

By replacing $\log(1+t)$ with $-L(a,-t)$ in \cref{definition:Generalized-Gregory}, we introduce the polynomial generalization of the Gregory coefficient as follows.
\begin{definition}
	We define the \emph{Generalized Gregory polynomials} $G_{m,n}(a)$ by the generating series
	\[
		\mathcal{G}(x,y; a) \coloneqq \sum_{m,n \ge 0} G_{m,n}(a) x^m y^n \coloneqq \frac{y L(a,-x)^2 - x L(a,-y)^2}{-L(a,-x) + L(a,-y)}.
	\]
\end{definition}
The first few examples are listed below.
\renewcommand{\arraystretch}{1.5}
\begin{table}[H]
\centering
\begin{tabular}{c||c|c|c|c}
	$m \backslash n$ & 0 & 1 & 2 & 3 \\ \hline \hline
	0 & 0 & 0 & 0 & $0$ \\ \hline	
	1 & 0 & 1 & $-a + \frac{1}{2}$ & $\frac{3}{2}a^2 - \frac{3}{2}a + \frac{1}{3}$  \\ \hline	
	2 & 0 & $-a + \frac{1}{2}$ & $\frac{1}{2}a^2 - \frac{1}{2}a + \frac{1}{12}$ & $-\frac{2}{3} a^3 + a^2 - \frac{5}{12}a + \frac{1}{24}$  \\ \hline	
	3 & 0 & $\frac{3}{2}a^2 - \frac{3}{2}a + \frac{1}{3}$ & $-\frac{2}{3}a^3 + a^2 - \frac{5}{12}a + \frac{1}{24}$ & $\frac{7}{8}a^4 - \frac{7}{4}a^3 + \frac{7}{6} a^2 - \frac{7}{24}a + \frac{7}{360}$ 
\end{tabular}
\caption{Examples of $G_{m,n}(a)$.}
\end{table}
In this case as well, the symmetric property $G_{m,n}(a) = G_{n,m}(a)$ holds by definition, and we can observe that $C_{i,r}(a) = G_{i,r-i+2}(a)$. As the final result, we show the coincidence.
\begin{theorem}[A generalization of \cref{Coefficient=Generalized-Gregory}]\label{Cir(a)=Gir(a)}
	For $1 \leq i \leq r$, we have $C_{i,r}(a) = G_{i,r-i+2}(a)$.
\end{theorem}
To prove this theorem, we need to generalize some lemmas.
\begin{lemma}[A generalization of \cref{Generalized-Gregory-Gregory}]
	We have
	\[
		\sum_{m=1}^\infty G_{m,1}(a) x^m = -L(a,-x) = \sum_{m=1}^\infty (-1)^{m-1} \lambda_m(a) x^m
	\]
	and
	\[
		1 - \sum_{m=1}^\infty G_{m,2}(a) x^m = \frac{x}{-L(a,-x)}.
	\]
\end{lemma}
\begin{proof}
	By taking a derivative of both sides of $e^{L(a,t)} (1 - te^{(a-1)L(a,t)}) = 1$ with respect to $t$, we have
	\[
		L_t(a,t) = \frac{1}{e^{(1-a)L(a,t)} - at}.
	\]
	Then, the claims follow from $\mathcal{G}_y(x,0; a)$ and $\mathcal{G}_{yy}(x,0; a)/2$.
\end{proof}

\begin{lemma}[A generalization of \cref{G-til-symmetry}]\label{Generalization of Gtil}
	For $m, n \geq 1$, we define
	\[
		\widetilde{G}_{m,n}(a) := G_{m,n}(a) + \sum_{j=0}^{m-2} G_{n-1, j+2}(a) \cdot G_{m-j-1, 2}(a).
	\]
	Then we have
	\begin{align*}
		\widetilde{\mathcal{G}}(x,y; a) &\coloneqq \sum_{m=1}^\infty \sum_{n=1}^\infty \widetilde{G}_{m,n+1}(a) x^m y^n\\
			&= \frac{-y L(a,-x) + x L(a,-y)}{-L(a,-x) + L(a,-y)} \left(\frac{-y L(a,-x) - x L(a,-y)}{xy} - 1 \right),
	\end{align*}
	which implies the symmetric property $\widetilde{G}_{m,n}(a) = \widetilde{G}_{n-1, m+1}(a)$ for $m \geq 1$ and $n \geq 2$.
\end{lemma}
\begin{proof}
	The proof is exactly the same as that of \cref{G-til-symmetry} by replacing $\log(1+t)$ with $-L(a,-t)$.
\end{proof}
As in the previous case, the symmetric property of $\widetilde{G}_{m,n}(a)$ captures the recurrence formula for $C_{i,r}(a)$. The remaining task is to check the coincidence of the initial values.
\begin{theorem}\label{GGP-initial-condition}
	For any $r \geq 1$, we have $C_{1,r}(a) = G_{1,r+1}(a)$, that is,
	\[
		(-1)^r \sum_{(n_1, \dots, n_r) \in S_r} \frac{B_{n_1}(a)}{n_1!} \cdots \frac{B_{n_r}(a)}{n_r!} = (-1)^r \lambda_{r+1}(a),
	\]
	where
	\[
		S_r \coloneqq S_{1,r,r} = \left\{ (n_1, \dots, n_r) \in \Z_{\geq 0}^r : 
		\begin{aligned}  
			&n_1 + \cdots + n_r = r,\\
			&n_{j+1} + \cdots + n_r \leq r-j\quad (1 \leq j < r)
		\end{aligned} \right\}.
	\]
\end{theorem}
To prove the theorem, we prepare some lemmas. 
\begin{lemma}\label{lambda-rec}
	For $n \geq 1$, we have
	\[
		\lambda'_n(a) = \sum_{k=1}^{n-1} k \lambda_{n-k}(a) \lambda_k(a),
	\]
	$\lambda_n(0) = (-1)^{n-1}/n$, and $\lambda_1(a) = 1$.
\end{lemma}
\begin{proof}
	By definition, we have
	\[
		L_a(a,t) = \frac{t L(a,t)}{e^{(1-a) L(a,t)} - at}, \qquad L_t(a,t) = \frac{1}{e^{(1-a)L(a,t)} - at}
	\]
	which implies that
	\[
		L_a(a,t) = tL_t(a,t)L(a,t).
	\]
	By comparing the coefficients on both sides, we obtain the first claim. The second claim follows from $L(0,t) = \log(1+t)$.
\end{proof}

\begin{lemma}\label{C(a)-der-rec}
	For any $1 \leq l \leq r$, we have
	\[
		\frac{\partial}{\partial a_l} C_{1,r}(a_1, \dots, a_r) = -\sum_{k=l}^r C_{1,k-1}(a_1, \dots, a_{k-1}) C_{1,r-k}(a_{k+1}, \dots, a_r),
	\]
	which implies that
	\[
		C'_{1,r}(a) = -\sum_{l=1}^r \sum_{k=l}^r C_{1,k-1}(a) C_{1,r-k}(a) = -\sum_{k=1}^r k C_{1,r-k}(a) C_{1,k-1}(a),
	\]
	where we put $C_{1,0}(a) \coloneqq 1$.
\end{lemma}
\begin{proof}
	For $1 \leq k \leq r$, we define
	\[
		\overline{S}_{k,r}  \coloneqq  \left\{ (n_1, \dots, n_r) \in \Z_{\geq 0}^r : 
		\begin{aligned}  
			&n_1 + \cdots + n_r = r-1,\\
			&n_{j+1} + \cdots + n_r \leq r-j-1 \quad (1 \leq j < k),\\
			&n_{j+1} + \cdots + n_r \leq r-j\quad (k \leq j < r)
		\end{aligned} \right\}.
	\]
	Then we see that $\overline{S}_{r,r} \subset \overline{S}_{r-1,r} \subset \cdots \subset \overline{S}_{1,r}$. First, for any $1 \le l \le r$, we can easily check that the map $\{(n_1, \dots, n_r) \in S_r : n_l > 0\} \to \overline{S}_{l,r}$ defined by
	\[
		(n_1, \dots, n_r) \mapsto (n_1, \dots, n_{l-1}, n_l - 1, n_{l+1}, \dots, n_r)
	\]
	is bijective. Second, for $1 \leq l < r$, we show that the map $\overline{S}_{l, r} \setminus \overline{S}_{l+1,r} \to S_{l-1} \times S_{r-l}$ defined by
	\[
		(n_1, \dots, n_r) \mapsto ((n_1, \dots, n_{l-1}), (n_{l+1}, \dots, n_r))
	\]
	is bijective. Since any $(n_1, \dots, n_r) \in \overline{S}_{l, r} \setminus \overline{S}_{l+1,r}$ satisfies that $n_l + \cdots + n_r \leq r-l$ and $r-l-1 < n_{l+1} + \cdots + n_r \leq r-l$, we have $n_l = 0$ and $n_{l+1} + \cdots + n_r = r-l$. Thus, the map is well-defined and we can easily check that it is bijective. Therefore, by the fact that $B'_n(a) = n B_{n-1}(a)$ and $B_0(a) = 1$, we have
	\begin{align*}
		\frac{\partial}{\partial a_l} &C_{1,r}(a_1, \dots, a_r) = (-1)^r \sum_{(n_1, \dots, n_r) \in \overline{S}_{l,r}} \frac{B_{n_1}(a_1)}{n_1!} \cdots \frac{B_{n_r}(a_r)}{n_r!}\\
			&= (-1)^r \sum_{(n_1, \dots, n_r) \in \overline{S}_{l+1,r}} \frac{B_{n_1}(a_1)}{n_1!} \cdots \frac{B_{n_r}(a_r)}{n_r!} - C_{1,l-1}(a_1, \dots, a_{l-1}) C_{1,r-l}(a_{l+1}, \dots, a_r)\\
			&= \cdots = (-1)^r \sum_{(n_1, \dots, n_r) \in \overline{S}_{r,r}} \frac{B_{n_1}(a_1)}{n_1!} \cdots \frac{B_{n_r}(a_r)}{n_r!} - \sum_{k=l}^{r-1} C_{1,k-1}(a_1, \dots a_{k-1}) C_{1,r-k} (a_{k+1}, \dots, a_r).
	\end{align*}
	Finally, since the map $\overline{S}_{r,r} \to S_{r-1}$ defined by
	\[
		(n_1, \dots, n_r) \mapsto (n_1, \dots, n_{r-1})
	\]
	is bijective, we have
	\[
		(-1)^r \sum_{(n_1, \dots, n_r) \in \overline{S}_{r,r}} \frac{B_{n_1}(a_1)}{n_1!} \cdots \frac{B_{n_r}(a_r)}{n_r!} = - C_{1,r-1}(a_1, \dots, a_{r-1}),
	\]
	which concludes the proof.
\end{proof}

\begin{proof}[Proof of \cref{GGP-initial-condition}]
	By comparing \cref{lambda-rec} with \cref{C(a)-der-rec}, we see that $C_{1,r}(a)$ and $G_{1,r+1}(a) = (-1)^r \lambda_{r+1}(a)$ satisfy the same recurrence relation. Here, the initial conditions that we have to check are $C_{1,0}(a) = \lambda_1(a)$ and $C_{1,r}(0) = (-1)^r \lambda_{r+1}(0)$. The first condition directly follows. The second condition can be checked by \cref{initial-condition} and $B_n(0) = (-1)^n B_n$. More precisely, we have
	\begin{align*}
		C_{1,r}(0) = (-1)^r C_{1,r} = \frac{1}{r+1} = (-1)^r \lambda_{r+1}(0),
	\end{align*}
	which concludes the proof.
\end{proof}

\begin{proof}[Proof of \cref{Cir(a)=Gir(a)}]
	By referring to \cref{Cir(a)-rec}, it suffices to show that $C_{1,r}(a) = G_{1,r+1}(a)$ and $\widetilde{G}_{i,j}(a) = \widetilde{G}_{j-1,i+1}(a)$ for $i \geq 1$ and $j \geq 2$, which are already shown in \cref{GGP-initial-condition} and \cref{Generalization of Gtil}.
\end{proof}
This concludes the present article, and it is worth noting that asymptotic coefficients beyond the origin were not addressed in this article. 
Exploring such coefficients is a potential topic for future research.

\section*{Acknowledgements}

This work was supported by JSPS KAKENHI Grant Numbers JP20K14292, JP21K18141 (Matsusaka), JP22K13897 (Murahara), and JP19K14511 (Onozuka). 
In addition, this work was supported by the Research Institute for Mathematical Sciences, an International Joint Usage/Research Center located in Kyoto University.

\bibliographystyle{amsplain}
\bibliography{References} 

\end{document}